\documentclass[11pt]{article}
\title{\LARGE On Perrin Cordial Labeling: A New Approach in Graph Labeling Theory}

\author{Sarbari Mitra, Soumya Bhoumik\\
Department of Mathematics\\
Fort Hays State University
}
\usepackage{amssymb}
\usepackage{amsmath}
\usepackage{algorithm}
\usepackage[noend]{algpseudocode}
\usepackage{amssymb}
\usepackage{array}
\usepackage{subcaption}
\usepackage{theorem}
\usepackage{graphicx}
\usepackage{cite}
\usepackage{colortbl}
\newtheorem{thrm}{Theorem}[section]

\newtheorem{lem}[thrm]{Lemma}
\newtheorem{cor}[thrm]{Corollary}
\newtheorem{prop}[thrm]{Proposition}

\usepackage{amsfonts}
\usepackage{graphicx}
\usepackage{graphics}
\usepackage{amsfonts}
\usepackage{multirow}
\usepackage{xmpmulti}

\usepackage{xmpmulti}
\usepackage[all,arc,curve,color,frame]{xy}

\theorembodyfont{\rm} \theoremstyle{definition}
\newtheorem{defin}[thrm]{Definition}

 \newcounter{case}

 \renewcommand{\thecase}{\arabic{case}}

\newcounter{subcase}

 \renewcommand{\thesubcase}{\alph{subcase}}

\usepackage{fullpage}

\def\mod{{\rm mod\ }}

\def\TS{{\rm TS}}

\newcounter{cases}
\newcounter{subcases}

\newenvironment{proof}{\noindent {\sc Proof}.}
                {\phantom{a} \hfill \framebox[2.2mm]{ } \bigskip}

\makeatletter
\def\BState{\State\hskip-\ALG@thistlm}
\makeatother

\providecommand{\keywords}[1]{\textbf{\textit{Keywords.}} #1}

\begin{document}

\pagestyle{plain}

\baselineskip = 1.2\normalbaselineskip

\maketitle

\begin{abstract}
In this paper, we introduce the concept of \emph{Perrin cordial labeling}, a novel vertex labeling scheme inspired by the Perrin number sequence and situated within the broader framework of graph labeling theory. The Perrin numbers are defined recursively by the relation \( P_n = P_{n-2} + P_{n-3} \), with initial values \( P_0 = 0 \), \( P_1 = 3 \), and \( P_2 = 0 \). A Perrin cordial labeling of a graph \( G = (V, E) \) is an injective function \( f : V(G) \rightarrow \{P_0, P_1, \dots, P_n\} \), where the induced edge labeling \( f^* : E(G) \rightarrow \{0,1\} \) is given by \( f^*(uv) = (f(u) + f(v)) \pmod 2 \). The labeling is said to be cordial if the number of edges labeled \( 0 \), denoted \( e_f(0) \), and the number labeled \( 1 \), denoted \( e_f(1) \), satisfy the condition \( |e_f(0) - e_f(1)| \leq 1 \). A graph that admits such a labeling is called a \emph{Perrin cordial graph}. This study investigates the existence of Perrin cordial labelings in various families of graphs by analyzing their structural properties and compatibility with the proposed labeling scheme. Our results aim to enrich the theory of graph labelings and highlight a new connection between number theory and graph structures.
\end{abstract}


\section{Introduction}
The domain of graph labeling has always fascinated researchers over the past few decades. Within this realm, a graph earns the designation of being cordial if it can be labeled with 0s and 1s in such a way that the difference of labels at the endpoints of its edges results in a distribution where the disparity between edges, labeled with ones and zeros, is at most one. 


\begin{defin}
A function $f:V(G)\rightarrow \{0,1\}$ is said to be Cordial Labeling if the induced function $f^*:E(G)\rightarrow \{0,1\}$ defined by $$f^*(uv)=\vert f(u)-f(v)\vert $$
satisfies the conditions $\vert v_f(0)-v_f(1)\vert \le 1$, as well as $\vert e_f(0)-e_f(1)\vert \le 1$, where \\
$v_f(0):=$ number of vertices with label $0$, \\
$v_f(1):=$ number of vertices with label $1$, \\ 
$e_f(0):=$ number of edges with label $0$, \\
$e_f(1):=$ number of edges with label $1$. 
\end{defin}

The concept of cordial labeling was introduced by Cahit in 1987 as a relaxed alternative to graceful and harmonious labeling. Later, Rokad and Ghodasara proposed Fibonacci cordial labeling \cite{RokadG2016}, and identified several graph families that admit such labelings. This framework was subsequently extended to additional classes of graphs (see \cite{MitraB2020, MitraPB2025, Rokad2017}). Motivated by these developments, we previously investigated cordial labeling schemes based on the Tribonacci sequence \cite{MitraB2022}.

Building upon this line of research, the present paper introduces a novel labeling scheme called Perrin cordial labeling, a natural extension of Fibonacci cordial labeling that draws upon the structure of the Perrin number sequence. In this work, we examine whether several well-known graph classes, such as paths, complete graphs, complete bipartite graphs, star graphs, wheel graphs, bistar graphs, triangular snake graphs, friendship graphs, and jellyfish graphs, include Perrin cordial labelings. The study also explores the behavior of this labeling scheme for larger instances within these families.

Although cordial labeling is less restrictive than graceful labeling, it still offers rich avenues for mathematical exploration. The integration of Perrin numbers into the labeling criteria inspired the development of Perrin cordial labeling, which bridges elements of graph theory and number theory. We believe that this framework offers a promising direction for future research at the intersection of these two domains.

\begin{defin}
The sequence of (reindexed) Perrin numbers $P_n$ is defined by the recurrence relation (for $n\ge 2)$:
$$P_{n} = P_{n-2} + P_{n-3};\ P_0=0, P_1=3, P_2=0, P_3=2, P_4=3, P_5=2, P_6=5,\cdots $$
\end{defin}

\begin{defin}
An injective function $f: V(G)\rightarrow \{P_0, P_1,\cdots, P_n\}$ is said to be Perrin cordial labeling if the induced function $f^*:E(G)\rightarrow \{0,1\}$ defined by $$f^*(uv)=(f(u)+f(v)) \ (\mod 2)$$
satisfies the condition $\vert e_f(0)-e_f(1)\vert \le 1$.
\end{defin}

In this paper, we denote the total number of even edges by $\varepsilon_0$ (analogously $\varepsilon_1$ for the number of odd edges). For any induced subgraph $H$ of $G$, then the number of even edges is denoted by $\varepsilon^H_0$, analogously $\varepsilon_0^H$, and $\varepsilon^H_0-\varepsilon^H_1$ will be denoted as $\tilde{\varepsilon}_H$ (we will consider $\tilde{\varepsilon}_G$ by $\tilde{\varepsilon}$).

\section{Main Results}
We start by observing that $k$ is the total number of even Perrin numbers in $P_n$. Then 
$$k =
\begin{cases}
3p+1, & \text{for } n=7p, 7p+1 \\ 
3p+2, & \text{for } n=7p+2 \\ 
3p+3, & \text{for } n=7p+3, 7p+4 \\ 
3p+4, & \text{for } n=7p+5, 7p+6 
\end{cases}
$$
We now investigate several families of graphs, starting with path graphs $P_n$, which admit Perrin cordial labelings. 

\begin{thrm}
All Paths are Perrin cordial. 
\end{thrm}

\begin{proof}



Since the parity of Perrin numbers (even or odd) is the only factor influencing the induced edge labels in Perrin cordial labeling, assigning specific Perrin numbers to vertices is unnecessary. Therefore, it suffices to consider the parity pattern rather than the individual values. In this paper, we establish that for any positive integer \( n \), the path graph \( P_n \) admits a Perrin cordial labeling. To begin, consider the case where \( n = 7p \) for some positive integer \( p \). In this setting, the labeling involves \( 3p + 1 \) even Perrin numbers, with the remaining \( 4p - 1 \) being odd.

We construct a labeling scheme as follows: assign the first \( q_1 \) vertices the odd Perrin numbers, followed by assigning \( p_1 \) vertices the even Perrin numbers. Then assign the next \( 2p_2 \) many vertices alternating odd and even Perrin numbers, starting with an odd one. Finally, assign the remaining \( q_2 \) vertices odd Perrin numbers. This structured assignment yields the imbalance,
\[
\varepsilon = (p_1 + q_1 + q_2 - 3) - (2p_2 + 2) = 7p - 4p_2 - 5,
\]
where we have used the relationships \( p_1 + p_2 = 3p \) and \( q_1 + q_2 + p_2 = 4p \), assuming without loss of generality that one even Perrin number is skipped. For the labeling to be Perrin cordial, we require \( |\varepsilon| = |7p - 4p_2 - 5| \leq 1 \), which holds when \( p \equiv 0, 2, 3 \pmod{4} \). 

On the other hand, if we choose \( q_1 = 0 \), then the imbalance becomes\(
 \varepsilon = 7p - 4p_2 - 3,
\)
which satisfies the Perrin cordial condition (i.e., \( |\varepsilon| \leq 1 \)) when \( p \equiv 1 \pmod{4} \). 

For all remaining cases on $n$ ($7p+q$ where $q\in\{1,2,\cdots,6\}$), Perrin cordiality can be verified by constructing and analyzing similar parity-based labeling.
\end{proof}

We now turn our attention to the cyclic graph $C_n$ and investigate its behavior under Perrin cordial labeling. To begin, we state the following lemma. The proof is omitted, as it follows directly from an argument analogous to that of Corollary 3.2 in \cite{MitraB2020}.

\begin{lem}\label{lem-cycle} 
For \textit{any} injective function $f:V(G)\rightarrow \{P_0,P_1,\cdots,P_{2m}\}$ on cyclic graph $C_{2m}$, if $ \varepsilon =2\ (\mod\ 4)$, then $C_{2m}$ can not be converted Perrin cordial.
\end{lem}

\begin{thrm}\label{cycle} 
The cyclic graph $C_{n}$ is Perrin cordial if and only if $n\not\equiv 2 \pmod 4$. 
\end{thrm}


\begin{proof}
First, observe that the cycle graph \( C_n \) is not Perrin cordial when \( n \equiv 2 \pmod{4} \). In particular, for any labeling of the vertices of \( C_{4m+2} \) with Perrin numbers, the resulting edge labeling yields an imbalance \( \varepsilon \equiv 2 \pmod{4} \). By Lemma~\ref{lem-cycle}, this violates the necessary condition for Perrin cordiality, and hence such a labeling cannot be transformed into a Perrin cordial labeling.

For the case where \( n \not\equiv 2 \pmod{4} \), we construct an explicit labeling of the vertices of the cyclic graph \( C_n \). Let \( V(C_n) = \{v_1, v_2, \ldots, v_n\} \) denote the vertices of the cycle. We assign {even Perrin numbers} to two distinct sets of vertices: the first \( p_1 \) vertices: \( \{v_1, v_2, \ldots, v_{p_1}\} \), and every second vertex after \( v_{p_1+1} \): \( \{v_{p_1+2}, v_{p_1+4}, \ldots, v_{p_1+2p_2}\} \). The remaining \( n - p_1 - p_2 \) vertices are assigned {odd Perrin numbers}. This labeling requires a total of \( p_1 + p_2 \) even Perrin numbers and \( n - (p_1 + p_2) + 1 \) odd Perrin numbers. 

Under this assignment, we obtain:
\[
\varepsilon_0 = n - 2p_2 - 2 \quad \text{and} \quad \varepsilon_1 = 2p_2 + 2,
\]
resulting in an imbalance of \( \varepsilon = n - 4p_2 - 4 \). To ensure \( \varepsilon \) is at most \( 1 \), we select \( p_2 = \frac{n - k}{4} \) for some \( k \in \{3, 4, 5\} \). This completes the proof.
\end{proof}

\begin{thrm}
The Complete graph $K_n$ is Perrin cordial iff $n=1,2,3,4,6,36,49,62,64,66,79,81,\\83$.
\end{thrm}

\begin{proof}
First note that the vertex labeling can be chosen from $\{P_0,P_1,\cdots,P_n\}$, which imply we drop either an odd or an even Perrin number from from list. First we assime that we use all the Perrin numbers, except an even one. Let us also consider $n=7p$. As there are $3p$ many even and $4p$ many odd vertex labels, $e_f(1)=12p^2$, and $e_f(0)={3p \choose 2} +{4p \choose 2}$. Hence in order to Perrin cordial we must have $$\bigg\vert {3p \choose 2} +{4p \choose 2} - 12p^2\bigg\vert \le 1$$ It simplifies to $\vert p^2-7p\vert\le 2$, whose only non-trivial solution is $p=7$, i.e. $n=49$.

Now, if we consider all Perrin numbers, except an odd one, a similar argument leads us to the conclusion that $K_n$ is not Perrin cordial for any $n$. Considering other cases on $n$ ($7p+q$ where $q\in\{1,2,\cdots,6\}$), we get the complete list as stated above.  
\end{proof}

Next, we look into the Perrin cordiality of the family of Complete Bipartite graphs $K_{m,n}$. Without loss of generality, we can assume that $n\le m$. Let us denote the vertices of $K_{m,n}$ by $\{v_{11},v_{12},\cdots, v_{1m}\}\cup \{v_{21},v_{22},\cdots, v_{2n}\}$. In order to make $K_{m,n}$ Perrin cordial, let us label $p_i$ many vertices of $\{v_{ij}\}$ by even Perrin numbers and the rest by odd. As $e(0) = p_1p_2+(m-p_1)(n-p_2)$, and $\varepsilon_1 = p_1(n-p_2)+ p_1(n-p_2)$, on simplification, we get

\begin{equation}\label{1}
\varepsilon_0- \varepsilon_1= (m-2p_1)(n-2p_2)
\end{equation}
Thus, if $n$ is even, then by assigning half of $n$ vertices even and the rest odd Perrin labels, we assure that $\varepsilon_0- \varepsilon_1=0$, which makes $K_{m,n}$ Perrin cordial. Now, we are left to consider the case when $n$ is odd.

\begin{thrm}
Let \( K_{m,n} \) be the complete bipartite graph, where \( n \) is odd and \( m \) is even. Then \( K_{m,n} \) admits a Perrin cordial labeling if and only if \( m \leq 6n + 26 \), with the exception that \( m \neq 6n + 22 \).
\end{thrm}

\begin{proof}
Since \( m \) is even, a Perrin cordial labeling requires that at least \( m/2 \) even Perrin numbers be assigned to half of the \( m \) vertices in the corresponding partite set (see Equation~\eqref{1}). 

Suppose \( m = 6n + 22 \). Then the total number of labels needed is \( m + n = 7n + 22 \), which can be expressed as \( 7(n + 3) + 1 \). Observing the pattern in the Perrin sequence, this implies that only \( 3(n + 3) + 1 = 3n + 10 \) even Perrin numbers are available among the first \( m + n \) terms. However, the requirement for a Perrin cordial labeling is at least \( m/2 = 3n + 11 \) even labels—one more than what is available. Therefore, a Perrin cordial labeling is not possible in this case.

For all other values of \( m \leq 6n + 26 \), with \( m \neq 6n + 22 \), an analysis of the distribution of even Perrin numbers within the first \( m + n \) terms of the sequence shows that the required number of even labels is available. Hence, the stated condition on \( m \) is both necessary and sufficient for the existence of a Perrin cordial labeling.
\end{proof}


\begin{thrm}
Let \( m \) and \( n \) be odd positive integers. The complete bipartite graph $K_{m,n}$ is Perrin cordial if one of the following conditions holds:
$$m+n \le 
\begin{cases}
28, & \text{if } m+n \equiv 0 \  (\mod 7) \\ 
22, & \text{if } m+n \equiv 1 \  (\mod 7) \\
30, & \text{if } m+n \equiv 2 \  (\mod 7) \\
38, & \text{if } m+n \equiv 3 \  (\mod 7) \\
32, & \text{if } m+n \equiv 4 \  (\mod 7) \\
40, & \text{if } m+n \equiv 5 \  (\mod 7) \\
34, & \text{if } m+n \equiv 6 \  (\mod 7) 
\end{cases}
$$
\end{thrm}
\begin{proof}
When both \( m \) and \( n \) are odd, achieving a Perrin cordial labeling requires at least \( \frac{m-1}{2} \) and \( \frac{n-1}{2} \) even Perrin numbers to be assigned (see Equation~\eqref{1}).

Consider the case where \( m + n = 7p \) for some positive integer \( p \). In this scenario, the total number of required even Perrin numbers is
\[
\frac{m-1}{2} + \frac{n-1}{2} = \frac{m+n - 2}{2} = \frac{7p - 2}{2} = \frac{7p}{2} - 1.
\]
To ensure this labeling is possible, we must have
\[
\frac{7p}{2} - 1 \leq 3p + 1,
\]
which simplifies to \( p \leq 4 \), or equivalently, \( m + n \leq 28 \).

The bounds for the remaining congruence classes of \( (m + n) \pmod 7 \) are established through similar case-wise analysis, completing the proof.
\end{proof}

We can immediately derive the following result for Star graphs.
\begin{cor}
The star graph \( S_n \) is Perrin cordial for all $ n \in \{1, 2, \ldots, 32\} \setminus \{25\}$
\end{cor}

A wheel graph $W_n$ is a graph that contains a cycle of $n$ many vertices such that every vertex of the cycle is connected with another vertex, situated at the center, known as the apex (See Figure \ref{wheel}).

\begin{thrm}
$W_n$ is Perrin cordial.
\end{thrm} 
\begin{proof}
We begin by denoting the vertices of the wheel graph $W_n$, where $ V(W_n) = \{v\} \cup \{v_1, v_2, \ldots, v_n\},$ with $v$ denoting the central (hub) vertex, and $v_1, v_2, \ldots, v_n$ forming the cycle. We assign {even Perrin numbers} to the vertices $\{v_1, v_2, \ldots, v_{p_1}\}$. The next group of vertices, $\{v_{p_1+1}, v_{p_1+2}, \ldots, v_{p_1+2p_2}\}$ receives Perrin numbers with {alternating parity}, starting with an {odd} Perrin number. The remaining $n - p_1 - 2p_2$ vertices, together with the central hub vertex $v$, are assigned {odd Perrin numbers}. This assignment requires a total of $p_1 + p_2$ even Perrin numbers and $n - (p_1 + p_2) + 1$ odd Perrin numbers. This assignment yields  ${\varepsilon}= 2n-2p_1-6p_2-5$. Now, if we consider $n + 1 = 7p + k$, for $0 \le k \le 6$, then to construct such a labeling scheme, we define:
\[
p_1 =
\begin{cases}
p + 3, & \text{if } k \in \{0, 6\}, \\
p + k, & \text{otherwise},
\end{cases}
\qquad
p_2 =
\begin{cases}
2p - 2, & \text{if } k = 0, \\
2p, & \text{if } k = 6, \\
2p - 1, & \text{otherwise}.
\end{cases}
\]
With this choice of $p_1$  and $p_2$, the labeling ensures that the imbalance $\varepsilon$ satisfies $\vert \varepsilon \vert \le 1$. Thus, the construction meets the requirements for Perrin cordial labeling. The reader may verify the parity distribution to confirm that the edge labels are balanced as claimed, completing the proof.
\end{proof}

\begin{figure}[h!]
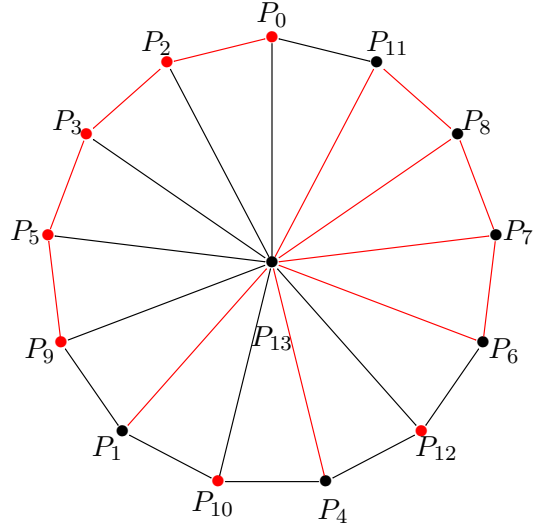

\[ \xygraph{
!{<0cm,0cm>;<0cm,1cm>:<-1cm,0cm>::}
!{(0,0);a(0)**{}?(0)}*{\bullet}="a13" !{(0,0);a(180)**{}?(1)}*{P_{13}}
!{(0,0);a(0)**{}?(3)}*[red]{\bullet}="a0" !{(0,0);a(0)**{}?(3.3)}*{P_{0}}
!{(0,0);a(27.7)**{}?(3)}*[red]{\bullet}="a12" !{(0,0);a(27.7)**{}?(3.3)}*{P_{2}}
!{(0,0);a(55.4)**{}?(3)}*[red]{\bullet}="a11" !{(0,0);a(55.4)**{}?(3.3)}*{P_{3}}
!{(0,0);a(83)**{}?(3)}*[red]{\bullet}="a10" !{(0,0);a(83)**{}?(3.3)}*{P_{5}}
!{(0,0);a(110.8)**{}?(3)}*[red]{\bullet}="a9" !{(0,0);a(110.8)**{}?(3.3)}*{P_{9}}
!{(0,0);a(138.5)**{}?(3)}*{\bullet}="a8" !{(0,0);a(138.5)**{}?(3.3)}*{P_{1}}
!{(0,0);a(166.2)**{}?(3)}*[red]{\bullet}="a7" !{(0,0);a(166.2)**{}?(3.3)}*{P_{10}}
!{(0,0);a(193.8)**{}?(3)}*{\bullet}="a6" !{(0,0);a(193.8)**{}?(3.4)}*{P_{4}}
!{(0,0);a(221.5)**{}?(3)}*[red]{\bullet}="a5" !{(0,0);a(221.5)**{}?(3.3)}*{P_{12}}
!{(0,0);a(249.2)**{}?(3)}*{\bullet}="a4" !{(0,0);a(249.2)**{}?(3.3)}*{P_{6}}
!{(0,0);a(277)**{}?(3)}*{\bullet}="a3" !{(0,0);a(277)**{}?(3.3)}*{P_{7}}
!{(0,0);a(304.6)**{}?(3)}*{\bullet}="a2" !{(0,0);a(304.6)**{}?(3.3)}*{P_{8}}
!{(0,0);a(332.3)**{}?(3)}*{\bullet}="a1" !{(0,0);a(332.3)**{}?(3.3)}*{P_{11}}
"a0"-@[red]"a12" "a12"-@[red]"a11" "a11"-@[red]"a10" "a10"-@[red]"a9" "a9"-"a8" "a8"-"a7"
"a7"-"a6" "a6"-"a5" "a5"-"a4" "a4"-@[red]"a3" "a3"-@[red]"a2" "a2"-@[red]"a1"
"a1"-"a0" 
"a0"-"a13" "a1"-@[red]"a13" "a2"-@[red]"a13" "a3"-@[red]"a13" "a4"-@[red]"a13" "a5"-"a13" "a6"-@[red]"a13" "a7"-"a13" "a8"-@[red]"a13" "a9"-"a13" "a10"-"a13" "a11"-"a13" "a12"-"a13" 
} 
\]\vspace{-.28 cm}
\caption{Perrin cordial labeling for $W_{14}$ graph}
\label{wheel}
\end{figure}


Next, we make the following observation regarding all possible edge compositions of a triangular blade (see Figure~\ref{oddeven}). In the illustration, red edges (respectively, black edges) represent edges labeled with even (respectively, odd) Perrin numbers, including both vertex and induced edge labels. Notably, within any triangle, the number of edges labeled with odd values is always either zero or two. The following proposition appears as a direct consequence of this observation. This result is generic in nature as it is independent of the underlying vertex labeling scheme.

\begin{prop}\label{Prop-odd-even}
    The number of odd edges on any blade $(C_3)$ is always even. 
\end{prop}

\begin{figure}[h!]
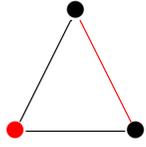
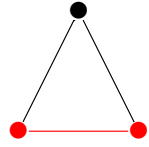
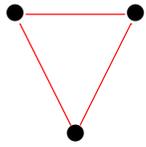
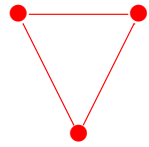

\centering
   \begin{subfigure}[b]{0.3\textwidth}
         \[ \xygraph{
!{<0cm,0cm>;<0cm,0.8cm>:<0.8cm,0cm>::}
!{(1,1)}*[red]{\scalebox{1.5}{$\bullet$}}="a0" 
!{(1,3)}*[black]{\scalebox{1.5}{$\bullet$}}="a1" 
!{(3,2)}*[black]{\scalebox{1.5}{$\bullet$}}="a2" 
"a1"-@[black]"a0"  "a0"-@[black]"a2" "a1"-@[red]"a2"
} 
\]    
     \subcaption{$1$ Even, $2$ Odds}
     \end{subfigure}
     \hspace{1 cm}
     \begin{subfigure}[b]{0.3\textwidth}
          \[ \xygraph{
!{<0cm,0cm>;<0cm,.8  cm>:<-.8 cm,0cm>::}
!{(1,1)}*[red]{\scalebox{1.5}{$\bullet$}}="a0" 
!{(1,3)}*[red]{\scalebox{1.5}{$\bullet$}}="a1" 
!{(3,2)}*[black]{\scalebox{1.5}{$\bullet$}}="a2" 
"a1"-@[red]"a0"  "a0"-@[black]"a2" "a1"-@[black]"a2"
} 
\]    
     \subcaption{$1$ Even, $2$ Odds }
     \end{subfigure}
     
      \vspace{0.2cm}
     \begin{subfigure}[b]{0.3\textwidth}
          \[ \xygraph{
!{<0cm,0cm>;<0cm,0.8 cm>:<-.8 cm,0cm>::}
!{(-1,-1)}*[black]{\scalebox{1.5}{$\bullet$}}="a0"
!{(-1,-3)}*[black]{\scalebox{1.5}{$\bullet$}}="a1" 
!{(-3,-2)}*[black]{\scalebox{1.5}{$\bullet$}}="a2"
"a1"-@[red]"a0"  "a0"-@[red]"a2" "a1"-@[red]"a2"
} 
\]    
     \subcaption{$3$ Even, $0$ Odds}
     \end{subfigure}
         \hspace{1 cm}
     \begin{subfigure}[b]{0.3\textwidth}
          \[ \xygraph{
!{<0cm,0cm>;<0cm,.8 cm>:<-0.8 cm,0cm>::}
!{(-1,-1)}*[red]{\scalebox{1.5}{$\bullet$}}="a0" 
!{(-1,-3)}*[red]{\scalebox{1.5}{$\bullet$}}="a1" 
!{(-3,-2)}*[red]{\scalebox{1.5}{$\bullet$}}="a2" 
"a1"-@[red]"a0"  "a0"-@[red]"a2" "a1"-@[red]"a2"
} 
\]    
     \subcaption{$3$ Even, $0$ Odds}
     \end{subfigure}
   \caption{Triangular blades have only two possible edge compositions in cordial labeling}
\label{oddeven}
\end{figure}
Next, we define the triangular snake graph $\TS_{n}$ in the following manner. We start with a path $P_{n+1}$ and for each pair of consecutive vertices of the path, attach a pendant vertex (see Figure~\ref{TS9}). We denote the vertex set of $\TS_{n}$ as $V(\TS_{n}) = \{v_1, v_2, \dots, v_{n+1}\} \cup \{u_1, u_2, \dots, u_n\}$, where the $v_i$ are the vertices of the underlying path and each $u_i$ is a vertex adjacent to both $v_i$ and $v_{i+1}$.
\begin{thrm}
The triangular snake graph $\TS_{n}$ satisfies the following Perrin cordiality conditions:
\begin{enumerate}
    \item $\TS_{n}$ is not Perrin cordial if $n\equiv 2 \ (\mod 4)$
    \item $\TS_{n}$ is Perrin cordial if $n\not \equiv 2 \ (\mod 4)$
\end{enumerate}
\end{thrm} 
\begin{proof}
First, observe that if $n \equiv 2 \pmod{4}$, then the total number of edges is $\vert E \vert = 12k + 6$, assuming $n = 4k + 2$ (for some non-negative $k$). To achieve a Perrin cordial labeling, we would require exactly $6k + 3$ odd-labeled edges among those contributed by the $n$ triangular blades. However, by Proposition~\ref{Prop-odd-even}, each triangular blade contributes an even number of odd edges, making it impossible to obtain an odd total count such as $6k + 3$. Hence, a Perrin cordial labeling cannot exist in this case.

We now show that the triangular snake graph $\TS_n$ admits a Perrin cordial labeling for all values of $n$ except when $n \equiv 2 \pmod{4}$. Consider assigning even Perrin numbers to the vertices in the sets $\{u_1, u_2, \cdots, u_{p_1}\} \cup \{u_{n+1-p_2}, u_{n+2-p_2}, \cdots, u_n\} \cup \{v_{n+1-p_2}, v_{n+2-p_2}, \cdots, v_{n+1}\}$, and assigning odd Perrin numbers to the remaining vertices. This labeling results in parameters $\varepsilon_0 = 3n - 2p_1 - 2$ and $\varepsilon_1 = 2p_1 + 2$.

Now, we construct a case-wise proof in this case. Any integer $n$ can be represented as $n=7p+q$ for some non-negative $p$ and $k\in\{0,1,\cdots,6\}$. We begin with $n = 7p$, for some positive integer $p$. Then, the number of even Perrin numbers used in this assignment is $p_1+2p_2+1=6p + 1$, and the number of odd Perrin numbers is $2n-p_1-2p_2=8p + 1$. If we skip one odd Perrin number, the resulting imbalance is $\varepsilon = 3n-4p_1-4=8p_2 - 3p - 4$, which remains within 1 in absolute value when $p \equiv 1, 4, 7 \pmod{8}$. On the other hand, if we skip one even Perrin number, we obtain $\varepsilon =8p_2 - 3p$, which stays within 1 when $p \equiv 0, 3, 5 \pmod{8}$. Together, these cover all values of $p$ except those that yield $n \equiv 2 \pmod{4}$. The remaining cases, where $n = 7p + q$ for $q \in {1, 2, \cdots, 6}$, can be addressed and verified using analogous arguments.
\end{proof}

\begin{figure}
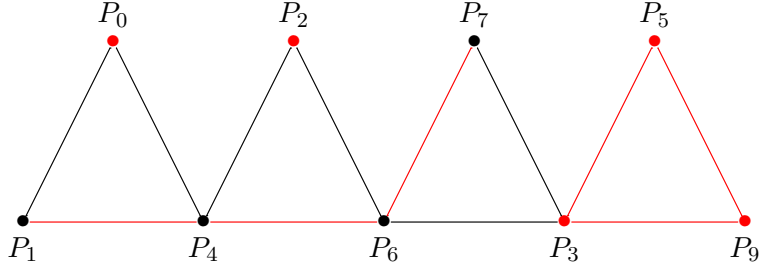

\[ \xygraph{
!{<0cm,0cm>;<1.2 cm,0cm>:<0cm,1.2 cm>::}
!{(-3,0)}*[red]{\bullet}="a0" !{(-3,0.3) }*{P_{0}}
!{(-1,0)}*[red]{\bullet}="a1" !{(-1,0.3)}*{P_{2}}
!{(1,0)}*{\bullet}="a2" !{(1,0.3) }*{P_{7}}
!{(3,0)}*[red]{\bullet}="a3" !{(3,0.3) }*{P_{5}}
!{(-4,-2)}*{\bullet}="b0" !{(-4,-2.3) }*{P_{1}}
!{(-2,-2)}*{\bullet}="b1" !{(-2,-2.3)}*{P_{4}}
!{(0,-2)}*{\bullet}="b2" !{(0,-2.3) }*{P_{6}}
!{(2,-2)}*[red]{\bullet}="b3" !{(2,-2.3) }*{P_{3}}
!{(4,-2)}*[red]{\bullet}="b4" !{(4,-2.3)}*{P_{9}}
"a0"-"b0" "a0"-"b1" "a1"-"b1" "a1"-"b2" "a2"-@[red]"b2" "a2"-"b3"  "a3"-@[red]"b3" "a3"-@[red]"b4"  
"b0"-@[red]"b1" "b1"-@[red]"b2" "b2"-"b3" "b3"-@[red]"b4" 
} \]
\caption{Perrin cordial labeling of $\TS_{4}$ graph}
\label{TS9}
\end{figure}

\noindent The friendship graph $F_n$ consists of $n$ cycles of length three ($C_3$), all sharing a single common vertex (see Figure \ref{friendshipgraph}). For clarity, we refer to this common vertex as the apex, denoted by $v$, and each individual cycle as a blade of the graph. By construction, the graph has $2n + 1$ vertices and $3n$ edges. We define the vertex set as $V(F_n) = \{v\} \cup \{v_1, v_2, \dots, v_{2n}\}$, where the $i^{\text{th}}$ blade corresponds to the triangle $C_i = \{v, v_{2i-1}, v_{2i}, v\}$ for $1 \le i \le n$.

\begin{thrm}
The friendship graph $F_n$ satisfies the following Perrin cordiality conditions:
\begin{enumerate}
    \item $F_{n}$ is not Perrin cordial if $n\equiv 2 \ (\mod 4)$
    \item $F_{n}$ is Perrin cordial if $n\not \equiv 2 \ (\mod 4)$
\end{enumerate}
\end{thrm} 
\begin{proof}
Suppose $n \equiv 2 \pmod{4}$. Then, we may write $n = 4m + 2$ for some positive integer $m$, which implies that the graph contains $4m + 2$ triangular blades. Since each triangle contributes 3 edges, the total number of edges in $F_n$ is $3n = 12m + 6$. To satisfy the Perrin cordial condition, we would require exactly half of the edges (i.e., $6m + 3$) to be odd-labeled. However, this leads to a contradiction: by Proposition~\ref{Prop-odd-even}, each triangle contributes an even number of odd-labeled edges, making it impossible to reach an odd total such as $6m + 3$. Therefore, $F_n$ cannot be Perrin cordial when $n \equiv 2 \pmod{4}$. 

Conversely, for all other values of $n$, we construct a Perrin cordial labeling using a strategy similar to the one described earlier. Specifically, we assign even Perrin numbers to the vertices in the sets $\{v_1, v_2, \dots, v_{2p_1}\} \cup \{v_{2p_1+1}, v_{2p_1+3}, \dots, v_{2p_1+2p_2-1}\}$ and assign odd Perrin numbers to the remaining vertices. This labeling yields the parameters $\varepsilon_0 = 3n - 2p_1 - 2p_2$ and $\varepsilon_1 = 2p_1 + 2p_2$.

As before, consider $n = 7p$ for some positive integer $p$. Under this assignment, the number of even and odd Perrin numbers used are $2p_1+p_2=6p + 1$ and $p_2+2(n-p_1-p_2)+1=8p + 1$, respectively. Without loss of generality, we skip one odd Perrin number, and the resulting imbalance becomes $\varepsilon = 3n-4p_1-4p_2=4p_1 - 3p - 4$, which lies within 1 in absolute value when $p \equiv 0, 1, 3 \pmod{4}$. It is worth noting that when $p \equiv 2 \pmod{4}$, we recover the excluded case $n \equiv 2 \pmod{4}$. The remaining cases, where $n = 7p + q$ for $q \in {1, 2, \dots, 6}$, can be verified through similar reasoning. 
\end{proof}

\begin{figure}[h!]
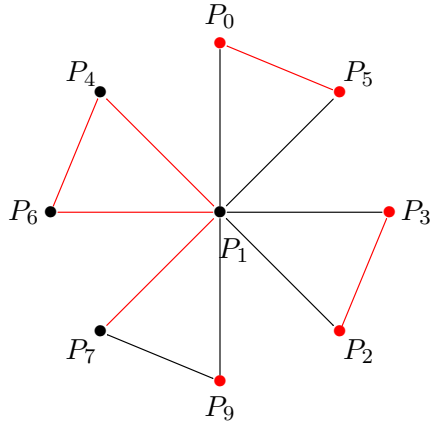

\[
\xygraph{
!{<0cm,0cm>;<0cm,.9 cm>:<-.9 cm,0cm>::}
!{(0,0);a(0)**{}?(0)}*{\bullet}="b" !{(0,0);a(200)**{}?(.6)}*{P_{1}}
!{(0,0);a(0)**{}?(2.5)}*[red]{\bullet}="a0" !{(0,0);a(0)**{}?(2.9)}*{P_{0}}
!{(0,0);a(45)**{}?(2.5)}*{\bullet}="a7" !{(0,0);a(45)**{}?(2.9)}*{P_{4}}
!{(0,0);a(90)**{}?(2.5)}*{\bullet}="a6" !{(0,0);a(90)**{}?(2.9)}*{P_{6}}
!{(0,0);a(135)**{}?(2.5)}*{\bullet}="a5" !{(0,0);a(135)**{}?(2.9)}*{P_{7}}
!{(0,0);a(180)**{}?(2.5)}*[red]{\bullet}="a4" !{(0,0);a(180)**{}?(2.9)}*{P_{9}}
!{(0,0);a(225)**{}?(2.5)}*[red]{\bullet}="a3" !{(0,0);a(225)**{}?(2.9)}*{P_{2}}
!{(0,0);a(270)**{}?(2.5)}*[red]{\bullet}="a2" !{(0,0);a(270)**{}?(2.9)}*{P_{3}}
!{(0,0);a(315)**{}?(2.5)}*[red]{\bullet}="a1" !{(0,0);a(315)**{}?(2.9)}*{P_{5}}
"a0"-@[red]"a1" 
"a2"-@[red]"a3" 
"a4"-"a5" 
"a6"-@[red]"a7"
"b"-"a0" 
"b"-"a1" 
"b"-"a2" 
"b"-"a3" 
"b"-"a4" 
"b"-@[red]"a5" 
"b"-@[red]"a6" 
"b"-@[red]"a7" 
} 
\]
\caption{Perrin cordial labeling for $F_4$ graph}
\label{friendshipgraph}
\end{figure}

Bistar $B_{m,n}$ is the graph obtained by joining the apex vertices of two copies of star, viz.  $K_{1,m}$ and $K_{1,n}$, by an edge. We identify the vertex set as $\{u_i:1\le i\le m\}\cup \{v_i:1\le i\le n\}\cup \{u,v\}$ where $u,v$ are the apex vertices and $u_i,v_i$ are the pendant vertices connected with $u$ and $v$ respectively. The vertex and edge set cardinalities are given by $m+n+2$ and $m+n+1$, respectively. 
\begin{thrm}
The bistar graph $B_{m,n}$ admits a Perrin cordial labeling if and only if $1<m+n\le 26$ or $m+n=28,29,30,32,36$.
\end{thrm}
\begin{proof}
Without loss of generality, we begin by assigning even Perrin numbers to the vertices in the sets $\{u_1, u_2, \dots, u_{p_1}\} \cup \{v_1, v_2, \dots, v_{p_2}\}$, while assigning odd Perrin numbers to the remaining vertices. This results in the use of $(m+n)-(p_1+p_2)+2$ odd Perrin numbers and $p_1+p_2$ even Perrin numbers, yielding the parameters $\varepsilon_0 = (m+n)-(p_1+p_2)+1$ and $\varepsilon_1 = p_1 + p_2$. Now, consider the case where $m+n+2 = 7p$ for some positive integer $p$; this implies $p_1 + p_2 = 3p$ and $m+n = 7p - 2$ when an even Perrin number is skipped. In this scenario, the imbalance $\vert \varepsilon \vert = \vert p - 1 \vert$ is less than $2$ only if $m+n = 12$ or $5$. Similarly, if we skip an odd Perrin number, then $p_1 + p_2 = 3p + 1$ (and $m+n = 7p - 2$ again), and the imbalance becomes $\vert \varepsilon \vert = \vert p - 3 \vert$, which is less than $2$ only when $m+n = 12$, $19$, or $26$. Considering the other cases ($m+n+2=7p+k$ for $1\le k\le 6$), we obtain the rest of the possible values for $m+n$ within the desired range.
\end{proof}

Finally, let us consider the jellyfish graph  $J_{m_1,m_2}$, defined as follows. The vertices $v_1,v_3,v_2,v_4,v_1$ form a $4-$cycle, with an additional chord between $v_3$ and $v_4$. Let $\{u^i_{j}:1\le j\le m_i\}$ denote pendant vertices, each connected to $v_i$ for $i=1,2$ (see Figure \ref{Jellyfish}). 
\begin{figure}
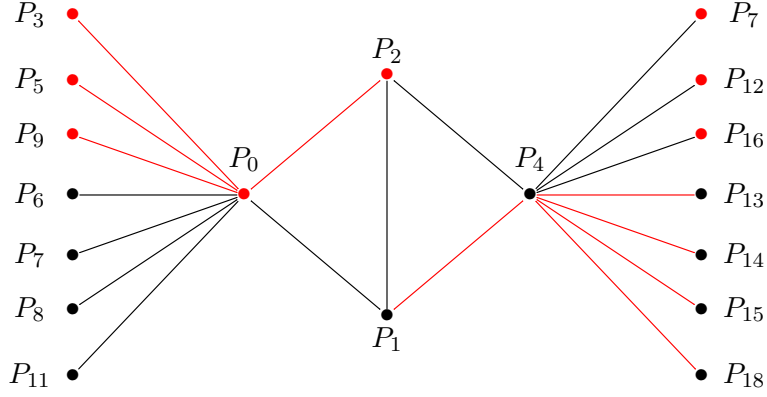

\[
\xygraph{
!{<0cm,0cm>;<1.9 cm,0cm>:<0cm,1.6cm>::}
!{(0,1)}*[red]{\bullet}="1" !{(0,1.2) }*{P_2}
!{(0,-1)}*{\bullet}="2" !{(0,-1.2) }*{P_1}
!{(1,0)}*{\bullet}="3" !{(1,0.3) }*{P_4}
!{(-1,0)}*[red]{\bullet}="4" !{(-1,0.3) }*{P_0}
!{(2.2,0.95)}*[red]{\bullet}="3a" !{(2.5,0.95)}*{P_{12}}
!{(2.2,0.5)}*[red]{\bullet}="3b" !{(2.5,0.5)}*{P_{16}}
!{(2.2,-0.5)}*{\bullet}="3c" !{(2.5,-0.5)}*{P_{14}}
!{(2.2,-0.95)}*{\bullet}="3d" !{(2.5,-0.95)}*{P_{15}}
!{(2.2,-1.5)}*{\bullet}="3e" !{(2.5,-1.5)}*{P_{18}}
!{(2.2,1.5)}*[red]{\bullet}="3f" !{(2.5,1.5)}*{P_{7}}
!{(2.2,0)}*{\bullet}="3g" !{(2.5,0)}*{P_{13}}
!{(-2.2,1.5)}*[red]{\bullet}="4a" !{(-2.5,1.5)}*{P_{3}}
!{(-2.2,-0.95)}*{\bullet}="4b" !{(-2.5,-0.95)}*{P_{8}}
!{(-2.2,-0.5)}*{\bullet}="4c" !{(-2.5,-0.5)}*{P_{7}}
!{(-2.2,-0)}*{\bullet}="4d" !{(-2.5,-0)}*{P_{6}}
!{(-2.2,0.5)}*[red]{\bullet}="4e" !{(-2.5,0.5)}*{P_{9}}
!{(-2.2,0.95)}*[red]{\bullet}="4f" !{(-2.5,0.95)}*{P_{5}}
!{(-2.2,-1.5)}*{\bullet}="4g" !{(-2.5,-1.5)}*{P_{11}}
"1"-"2" "2"-@[red]"3" "2"-"4" "4"-@[red]"1" "1"-"3"
"3"-"3a" "3"-"3b" "3"-@[red]"3c" "3"-@[red]"3d" "3"-@[red]"3e" "3"-"3f" "3"-@[red]"3g"
"4"-@[red]"4a" "4"-"4b" "4"-"4c" "4"-"4d" "4"-@[red]"4e" "4"-@[red]"4f" "4"-"4g"
}
\]
\caption{Perrin cordial labeling for $J_{7,7}$}
\label{Jellyfish}
\end{figure}
\begin{thrm}
All Jellyfish graphs are Perrin cordial.
\end{thrm}
\begin{proof}
To prove that the jellyfish graph is Perrin cordial, we begin by specifying a labeling strategy for the vertices. For each \( i = 1, 2 \), let the pendant vertices \( \{u^i_1, u^i_2, \ldots, u^i_{k_i}\} \) be assigned even Perrin numbers, and the remaining pendant vertices \( \{u^i_{k_i+1}, u^i_{k_i+2}, \ldots, u^i_{m_i}\} \) be assigned odd Perrin numbers.

We first consider the case when both \( m_1 \) and \( m_2 \) satisfy \( m_1, m_2 \equiv 0 \pmod{7} \). Let \( m_i = 7p_i \) for integers \( p_1, p_2 \). Then, the total number of available even Perrin numbers is \( 3(p_1 + p_2) + 3 \), and the number of odd Perrin numbers is \( 4(p_1 + p_2) + 2 \). Without loss of generality, we skip one even Perrin number in our labeling to maintain balance.

Now consider the labeling of the internal vertices \( v_1, v_2, v_3, \) and \( v_4 \). Suppose both \( v_1 \) and \( v_3 \) are labeled with even Perrin numbers, and \( v_2 \) and \( v_4 \) with odd Perrin numbers. Then, the imbalance parameter is calculated as follows:
\begin{eqnarray*}
    \varepsilon & = & k_1 + (m_2 - k_2 + 2) - \left[(m_1 - k_1) + k_2 + 3\right] \\
                & = & (m_2 - m_1) + 2(k_1 - k_2) - 1 \\
                & = & 7p_2 - 7p_1 + 2(3(p_1 + p_2) - k_2 - k_2) - 1 \\
                & = & 13p_2 - p_1 - 4k_2 - 1.
\end{eqnarray*}

Alternatively, if only \( v_1 \) is labeled with an even Perrin number and \( v_2, v_3, v_4 \) are labeled with odd Perrin numbers, a similar calculation yields:
\[
\varepsilon = 13p_2 - p_1 - 4k_2 + 3.
\]

We adopt the first labeling scheme (both \( v_1 \) and \( v_3 \) labeled with even Perrin numbers) when \( p_2 - p_1 \equiv 0, 1, 2 \pmod{4} \), and the second scheme when \( p_2 - p_1 \equiv 3 \pmod{4} \). This guarantees that \( |\varepsilon| \leq 1 \) in all cases, ensuring that the labeling is Perrin cordial.

For all other cases (i.e., when either \( m_1 \not\equiv 0 \pmod{7} \) or \( m_2 \not\equiv 0 \pmod{7} \)), the proof follows similarly by adjusting the counts of available Perrin numbers accordingly. Thus, we conclude that the jellyfish graph is Perrin cordial in all cases.
\end{proof}

\section{Conclusion}

In this paper, we introduced and developed the concept of \textit{Perrin cordial labeling}, a new graph labeling scheme grounded in the Perrin number sequence. To the best of our knowledge, this is the first attempt to define the Perrin cordial labeling in the area of graph theory. By defining and applying this labeling to various standard graph families—including paths, cycles, complete and complete bipartite graphs, star graphs, wheel graphs, and jellyfish graphs—we have laid the foundational groundwork for this novel intersection between number theory and graph labeling. Our results show that several classical graph families admit Perrin cordial labelings under the defined criteria, highlighting both the versatility and constraints of this scheme. The Perrin cordial labeling presents rich potential for further exploration, especially for complex or compound graph structures. Future research directions may include the characterization of additional graph families—such as circulant graphs, corona products, and other composite structures—with respect to Perrin cordial labeling. Moreover, developing efficient algorithms for constructing such labelings remains an important open problem. Another promising avenue lies in exploring deeper theoretical connections between Perrin cordial labeling and other well-known numerical sequences, such as the Fibonacci and Tribonacci sequences, to uncover broader structural or algebraic relationships within graph labeling theory.

\bibliographystyle{amsplain}

\begin{thebibliography}{99}

\bibitem{BoeschW1985} F. T. Boesch and J. Wang, Reliable circulant networks with
minimum transmission delay, \emph{IEEE Transactions on Circuit and systems}, \textbf{32} (1985), 1286--1291.


\bibitem{MitraB2020} S. Mitra and S. Bhoumik, Fibonacci Cordial Labeling of Some Special Families of Graphs, \emph{Annals of Pure and Applied Mathematics}, \textbf{21-2} (2020) 135--140.


\bibitem{MitraB2022} S. Mitra and S. Bhoumik, Cordial Labeling of Graphs Using Tribonacci Numbers, \emph{International Journal of Mathematical Combinatorics}, \textbf{2}, (2022) 33--46.


\bibitem{MitraPB2025} S. Mitra, A. Pritchard and S. Bhoumik, On Fibonacci Cordial Labeling of Some Planar Graphs, \emph{TWMS Journal of Applied and Engineering Mathematics}, (Accepted).


\bibitem{Rokad2017} A. H. Rokad, Fibonacci Cordial Labeling of Some Special Graphs, \emph{Oriental Journal of Computer Science and Technology}, \textbf{10(4)} (2017) 824--828.

\bibitem{RokadG2016} A. H. Rokad and G. V. Ghodasara, Fibonacci Cordial Labeling of Some Special Graphs, \emph{Annals of Pure and Applied Mathematics}, \textbf{11(1)} (2016) 133--144.

\bibitem{Yeh2006} R. K. Yeh, A survey on labeling graphs with a condition at distance two, \emph{Discrete Mathematics}, \textbf{306} (2006) 1217--1231.


\bibitem{WongC1974} G. K. Wong and D. A. Coppersmith, A combinatorial problem related to multi module memory organization, \emph{Journal of Association for Computing Machinery}, \textbf{21} (1974), 392--401.




\end{thebibliography}

\end{document}